\documentclass[11pt]{amsart}
\usepackage[utf8]{inputenc}
\usepackage{amsmath,amssymb}
\usepackage{wrapfig}
\usepackage{url}
\usepackage{mathtools}
\usepackage{graphicx}
\usepackage{stmaryrd}
\usepackage{amsthm}
\usepackage{xcolor}
\usepackage[colorlinks=true,linkcolor=blue,citecolor=blue]{hyperref}
\usepackage[shortlabels]{enumitem}
\usepackage{comment}
\usepackage{relsize}
\usepackage{dsfont}
\usepackage{stmaryrd}
\usepackage{geometry}
\usepackage{setspace}
 \usepackage{relsize}
\usepackage{mathrsfs} 

\newtheorem{theorem}{Theorem}[section]
\newtheorem{corollary}[theorem]{Corollary}
\newtheorem{prop}[theorem]{Proposition}
\newtheorem{remark}[theorem]{Remark}

\newtheorem{lemma}[theorem]{Lemma}
\theoremstyle{definition}

\theoremstyle{definition}
\newtheorem{definition}[theorem]{Definition}
\DeclareMathOperator{\R}{\mathbb{R}}

\DeclareMathOperator{\N}{\mathbb{N}}
\DeclareMathOperator{\Real}{\text{Re}}
\DeclareMathOperator{\Tr}{Tr}
\DeclareMathOperator{\pseudoxi}{\langle\xi\rangle}

\date{\today}

\title[Vector Valued G\r{a}rding Inequality on compact Lie groups]{Vector Valued G\r{a}rding Inequality for Pseudo-differential Operators on Compact Homogeneous Manifolds}
\author{André Kowacs and Michael Ruzhansky}

\thanks{The authors are supported  by the FWO  Odysseus  1  grant  G.0H94.18N:  Analysis  and  Partial Differential Equations, by the Methusalem programme of the Ghent University Special Research Fund (BOF)
(Grant number 01M01021) and by the FWO grant G011522N. Michael Ruzhansky was also supported by EPSRC grant EP/V005529/1. Andr\'e Kowacs was supported in part by the Coordenação de Aperfeiçoamento de Pessoal de N\'ivel Superior - Brasil (CAPES) - Finance Code 001”;}

\subjclass{Primary 22E30, 43A77. Secondary 58J40, 35S10}

\keywords{G\aa rding inequality, Compact Lie group, Pseudo-differential operator, Fourier Analysis.}

\begin{document}

\maketitle

\begin{abstract}
    We prove sufficient conditions in order to obtain a sharp G\aa rding inequality for pseudo-differential operators acting on vector-valued functions on compact Lie groups. As a consequence, we obtain a sharp G\aa rding inequality for compact homogeneous vector bundles and compact homogeneous manifolds. The sharp G\aa rding inequality is the strongest lower bound estimate known to hold for systems on $\R^n$, and the aim of this paper is to extend this property to systems on compact Lie groups and compact homogeneous manifolds. Our results extend previous works by Lax and Nirenberg [Comm. Pure Appl. Math., Vol. 8, 129-209, (1966)], and by Ruzhansky and Turunen [J. Funct. Anal., Vol. 267, 144-172, (2011)]. As an application, we establish existence and uniqueness of solution to a class of systems of initial value problems of pseudo-differential equations on compact Lie groups and compact homogeneous manifolds.
\end{abstract}
{
\hypersetup{linkcolor=red}
\tableofcontents}
\section{Introduction}
The so called G\aa rding inequality was first proved by  G\aa rding in his paper \cite{Gaard} and can be stated as follows:\\
{\it Let $P$ be an elliptic self-adjoint pseudo-differential operator of order $m\in\R$ on an open set $U\subset \R^n$. Then, for any compact set $Q\subset U$ and $\gamma<m/2$, there exist constants $c_{\gamma,Q},\,C_{\gamma,Q}>0$ such that}
\begin{equation*}
    (Pu,u)_{L^2}\geq c_{\gamma,Q}\|u\|^2_{H^{m/2}}-C_{\gamma,Q}\|u\|^2_{H^{\gamma}},
\end{equation*}
for every $u\in C_0^\infty(Q)$.
    In order to apply it to different problems, Hörmander \cite{Hor} and Lax and Nirenberg \cite{Nire} then adapted this result and proved the so called sharp G\aa rding inequality for pseudo-differential operators for symbols in the Kohn-Nirenberg class $S^m(\R^{2n})$. Their result states that if a symbol $p\in S^m(\R^{2n})$, for $m\in\R$ satisfies $\Real p(x,\xi)\geq 0$, then the pseudo-differential operator $p(x,D)$ associated to this symbol satisfies the estimate
    \begin{equation*}
        \Real(p(x,D)u,u)_{L^2(\R^n)}\geq -C\|u\|_{H^{\frac{m-1}{2}}}^2,
    \end{equation*}
    for some $C>0$ and every $u\in H^{\frac{m-1}{2}}(\R^n)$, where the norm on the right hand side of the inequality denotes the usual norm in the Sobolev space $H^{\frac{m-1}{2}}(\R^n)$. Since then, this result has been extended, by many authors, to different settings and symbol classes. In particular, this type of inequality has been extended to the global symbol classes on compact Lie groups for scalar valued functions. In this paper, inspired by the works on G\aa rding inequality on compact Lie group \cite{RuzWirth}, 
    on sharp G\aa rding inequality on compact Lie groups \cite{RuzSharp}, 
    on sharp G\aa rding inequality for subelliptic operators on compact Lie groups \cite{SharpSub} and 
    on G\aa rding inequality on graded Lie groups \cite{DuvanJulioGraded}, we prove a type of sharp G\aa rding inequality for pseudo-differential operators acting on vector-valued functions on compact Lie groups. In order to do so, we also extend the notion of pseudo-differential operators defined by amplitudes to the vector-valued compact Lie group setting. As a consequence, we obtain a sharp G\aa rding inequality for pseudo-differential operators on compact homogeneous vector bundles and compact homogeneous manifolds. 
    Finally, in the end we present an application of these result proving existence and uniqueness of solution to a class of systems of vector valued Cauchy problems of pseudo-differential equations. It is worth mentioning that the symbol classes considered are the global symbol classes defined in \cite{RuzPseudo2010}. These classes of symbols coincide with the usual Hörmander local symbol classes on manifolds in the range that the latter are well defined, that is $0\leq \delta<\rho\leq 1$, $\rho\geq 1-\delta$. The global symbols for compact Lie groups developed in \cite{RuzPseudo2010} allow us to extend this range to include the cases $\rho<1-\delta$, and thus are more general to deal with. It is also relevant to mention that the sharp G\aa arding inequality is the strongest lower bound estimate known to hold for systems on $\R^n$ (vector-valued functions). The aim of this paper is to extend this property for global quantisation of operators on compact Lie groups.
    \\
    This paper is organized as follows:\\
    \indent In Section \ref{Prelim} we recall some of the theory about pseudo-differential operators on compact Lie groups. Next we also develop and prove some results concerning amplitudes for pseudo-differential operators acting on vector-valued functions on compact Lie groups. At the end of this section, we also recollect some basic facts about parity of functions on Lie groups.\\
    \indent In Section \ref{MainRes} state and prove our mains results. We also prove a few lemmas which aid in its proofs.\\
    \indent Finally, in Section \ref{PDE} we present an application of our main result, in which we prove existence and uniqueness of solution to a certain class of systems of pseudo-differential equations.
    
\section{Preliminaries}\label{Prelim}
\subsubsection{The group Fourier transform} Let $G$ be a compact Lie group. Denote by $\text{Rep}(G)$ the set of all continuous unitary representations of $G$, that is, the set of all continuous $\xi\in \text{Hom}(G,\text{Aut}(\mathcal{H}_\xi))$, where $\mathcal{H}_\xi$ is some Hilbert space and $\xi(g)$ is unitary for every $g\in G$. It is well known that for all $\xi\in\text{Rep}(G)$, the spaces $\mathcal{H}_\xi$ are finite dimensional, and we denote their dimensions by $d_\xi$. Define the equivalence relation $\xi\sim\eta$ if there exists a bijection $A:\mathcal{H}_\xi\to H_\eta$ such that $\eta(g)A=A\xi(g)$, for all $g\in G$. The set of all such equivalence classes, denoted by $\widehat{G}$, is known as the unitary dual of $G$. As every class $[\xi]\in\widehat{G}$ is finite dimensional, we can always choose a representative that is matrix-valued. It follows from the Peter-Weyl Theorem that the collection of all coefficient functions of elements of $\widehat{G}$ is an orthogonal basis for $L^2(G)$, where integration is taken with respect to the Haar measure in $G$. This motivates the definition of the matrix-valued Fourier coefficients of $f\in L^2(G)$ by 
\begin{equation*}
    \widehat{f}(\xi)\doteq \int_G f(x)\xi(x)^*dx,
\end{equation*}
for $\xi\in\text{Rep}(G)$. This gives the Fourier inversion formula as
\begin{equation*}
    f(x)=\sum_{[\xi]\in\widehat{G}} d_\xi\Tr\left(\xi(x)\widehat{f}(\xi)\right),
\end{equation*}
for almost every $x$ in $G$, where in the sum we choose a single representative for each class in $\widehat{G}$. As in the Euclidean case, there is also an analogue of Plancherel's Theorem, namely
\begin{equation*}
    \|f\|_{L^2(G)}=\left(\int_G|f(x)|^2dx\right)^{\frac{1}{2}}=\left(\sum_{[\xi]\in \widehat{G}}d_\xi\|\widehat{f}(\xi)\|_{HS}^2\right)^{\frac{1}{2}}=\|\widehat{f}\|_{L^2(\widehat{G})},
\end{equation*}
for any $f\in L^2(G)$, where $\|A\|_{HS}^2\doteq \Tr(A^*A)$ for any matrix $A$.
When $G$ is also a Lie group, these definitions can be extended to the set $\mathcal{D}'(G)$ of distributions on $G$, through the formula
\begin{equation*}
    \widehat{u}(\xi)\doteq\langle u,\xi^*\rangle,
\end{equation*}
for $u\in\mathcal{D}'(G)$ and $\xi\in \text{Rep}(G)$, where evaluation should be understood coefficient-wise.

\subsubsection{The quantisation on compact Lie groups}

Now assume that $G$ is also a Lie group. Given a continuous linear operator $L:C^\infty(G)\to C^\infty(G)$, define its matrix-valued global symbol $\sigma_L$ by
\begin{equation*}
    \sigma_L(x,\xi)\doteq \xi(x)^*L\xi(x),
\end{equation*}
for all $(x,[\xi])\in G\times\widehat{G}$. In \cite{RuzPseudo2010} the authors proved the quantisation formula
\begin{equation}\label{quantisation} Lf(x)=\sum_{[\xi]\in\widehat{G}}d_\xi\Tr\left[\xi(x)\sigma_L(x,\xi)\widehat{f}(\xi)\right]
\end{equation}
for all $x\in G$, $f\in C^\infty(G)$.
In order to classify such symbols, a difference operator of order $k\in\mathbb{N}$,  $\Delta_q$ on $\widehat{G}$ , where $q\in C^\infty(G)$ vanishes of order $k$ at the group identity $e=e_G$, is defined by 
\begin{equation*}
    \Delta_q \widehat{f}(\xi)\doteq \widehat{qf}(\xi),
\end{equation*}
for every $[\xi]\in\widehat{G}$.\\
A family of difference operators $\{\Delta_{q_1},\dots,\Delta_{q_r}\}$ of order $1$ is said to be strongly admissible if 
\begin{equation*}
    \text{rank}\{\nabla q_{j}(e)|1\leq j\leq r\}=\dim(G)
\end{equation*}
and
\begin{equation*}
    \bigcap_{j=1}^{r}\{x\in G| q_j(x)=0\}=\{e_G\}.
\end{equation*}
We fix such a family and define the multi-index notation for higher order difference operators by
\begin{equation*}
    q_{\alpha}\doteq q_{1}^{\alpha_1}\dots q_{k}^{\alpha_k},
\end{equation*}
\begin{equation*}
\Delta^\alpha_\xi\doteq\Delta_{q_1}^{\alpha_1}\dots \Delta_{q_r}^{\alpha_r}\equiv \Delta_{q_{\alpha}},
\end{equation*}
for a multi-index $\alpha$. We recall (see \cite{SharpSub}) that we may choose such family so that an analog of the Leibniz's rule holds as follows:

\begin{prop}[Leibniz's like formula for Difference Operators]\label{leibniz}
    For any multi-index $\alpha$ there exist constants $C_{\lambda,\mu}\geq 0$ such that
    \begin{equation*}
        \Delta_\xi^{\alpha}[\widehat{f}(\xi)\widehat{g}(\xi)]=\sum_{|\mu|,|\lambda|\leq |\alpha|\leq |\lambda+\mu|}C_{\lambda,\mu}(\Delta_\xi^{\lambda}\widehat{f}(\xi))(\Delta_\xi^{\mu}\widehat{g}(\xi)),
    \end{equation*}
    for any $f,g\in\mathcal{D}'(G)$, and all $[\xi]\in\widehat{G}$.
\end{prop}

Next, recall that the vector space of left-invariant vector fields of $G$, denoted by $\mathfrak{g}$, may be identified with the tangent space at the identity $e$, and each of its non-zero elements corresponds to a first-order differential operator $X:C^\infty(G)\to C^\infty(G)$ given by
\begin{equation*}
    Xf(x)\doteq \frac{d}{dt}\bigg|_{t=0}f(x\exp(tX)),
\end{equation*}
for $x\in G$, where $\exp:\mathfrak{g}\to G$ denotes the standard Lie exponential mapping. If $\{X_1,\dots,X_n\}$ is a basis for $\mathfrak{g}$, we will use the standard multi-index notation 
\begin{equation*}
    \partial^\alpha=X_1^{\alpha_1}\dots X_n^{\alpha_n},
\end{equation*}
for a standard left-invariant differential operator of order $|\alpha|$. Notice that, in general, these operators do not commute, hence to represent all differential operators of a certain order, one would need to specify their order of composition. As this particular order will not be relevant, we will consider $\partial^\alpha$ as denoting an arbitrary differential operator of order $|\alpha|$.\\
Consider now the positive Laplace-Beltrami operator of $G$, $\mathcal{L}_G$. This is a second order bi-invariant differential operator and the coefficient functions $\{\xi_{ij}\}_{i,j=1}^{d_\xi}$ of $[\xi]\in\widehat{G}$ are all eigenfunctions correponding to the same respective eigenvalue $\lambda_{[\xi]}\geq0$. Defining the weight
\begin{equation*}
    \langle\xi\rangle\doteq (1+\lambda_{[\xi]})^{1/2},
\end{equation*}
corresponding to the eigenvalues of the first order differential operator $(\text{Id}+\mathcal{L}_G)^{\frac{1}{2}}$. We use these to describe the symbol classes as follows.

\begin{definition}
Let $G$ be a compact Lie group, $m\in\R$, $0\leq \rho,\delta\leq 1$. We shall say that a matrix-valued function
\begin{equation*}
    \sigma:G\times\widehat{G}\to\bigcup_{[\xi]\in\widehat{G}}\mathbb{C}^{d_\xi\times d_\xi},
\end{equation*}
smooth in $x\in G$, is in the symbol class $\mathcal{S}_{\rho,\delta}^m(G)$ if
\begin{equation*}
    \|\partial^\beta_x\Delta^\alpha_\xi\sigma(x,\xi)\|_{op}\leq C_{\alpha,\beta}\langle\xi\rangle^{m-\rho|\alpha|+\delta|\beta|},
\end{equation*}
for every multi-index $\alpha,\beta$ and for all $(x,[\xi]) \in G\times\widehat{G}$. For $\sigma_L\in \mathcal{S}_{\rho,\delta}^m(G)$ we will write $L\in \Psi^m_{\rho,\delta}(G)\equiv\text{Op}(\sigma_L)$ for the continuous linear operator defined by formula \eqref{quantisation}.
\end{definition}
The operators associated with such symbol classes are called pseudo-differential operators, the real number $m$ is said to be its order. \\
As $G$ is in particular also a smooth manifold, one could have defined the Hormander classes of pseudo-differential operators by local coordinate systems, denoted $\Psi_{\rho,\delta}^m(G,\text{loc})$ for $m\in\R$, $0\leq\delta<\rho\leq 1$, $\rho\geq 1-\delta$. In fact, the first definition extends the second in the following sense.
\begin{theorem} [Equivalence of classes, \cite{RuzPseudo2010},\cite{VerIntri}]
Let $L:C^\infty(G)\to C^\infty(G)$ be a continuous linear operator and $m\in R$, $0\leq\delta<\rho\leq 1$, $\rho\geq 1-\delta$. Then, $L\in \Psi_{\rho,\delta}^m(G,\text{loc})$ if, and only if, $\sigma_L\in\mathcal{S}_{\rho,\delta}^m$. Consequently, 
\begin{equation*}
    \Psi_{\rho,\delta}^m(G,\text{loc})
    =\Psi^m_{\rho,\delta}(G),\, 0\leq\delta<\rho\leq 1,\, \rho\geq 1-\delta, \,m\in\R. 
\end{equation*}
\end{theorem}

In this paper, we shall make use of the following lemma, whose proof can be found in \cite{VerIntri}.

\begin{lemma}[Taylor Expansion in Compact Lie Groups]
    Let $G$ be a compact Lie group of dimension $d$, and let $\{\Delta_{q_1},\dots,\Delta_{q_{d}}\}$ be a strongly admissible collection of difference operators. Then there exists a basis of left-invariant vector fields $\{X_1,\dots,X_d\}$ such that
    \begin{equation*}
        X_jq_k(x^{-1})|_{x=e_G}=\delta_{jk},
    \end{equation*}
    for all $1\leq j,k\leq d$. Moreover, by using a multi-index notation relative to both this collection of difference operators, and this basis of left invariant vector fields, any $f\in C^\infty(G)$, can be written as
    \begin{equation*}
        f(xy)=f(x)+\sum_{1\leq |\alpha|< N} \frac{1}{\alpha!}\partial^{\alpha}f(x)q_{\alpha}(y^{-1}) + R_{x,N}^f(y),
    \end{equation*}
    for any $x,y\in G$, $N\in \N$, where the last term is referred to as the Taylor remainder of order $N$ and satisfies
    \begin{equation*}
        |R_{x,N}^f(y)|\lesssim |y|^N \max_{|\alpha|\leq N}\|\partial^\alpha f\|_{L^\infty(G)}.
    \end{equation*}
\end{lemma}

\subsection{The vector-valued quantisation on compact Lie groups}
Let $E_0$ be a $n$-dimensional $\mathbb{C}$-vector space, where $n\in \N$. Consider $B_0=\{e_1,\dots,e_n\}$ an orthonormal basis on $E_0$. We may identify $E_0\cong \mathbb{C}^n\cong \mathbb{C}^{n\times1}$ by identifying each element in $B_0$ with the canonical basis in $\mathbb{C}^n$. Hence, given a mapping $f:G\to E_0$, we may write it as 
\begin{equation*}
    f(x)=(f_1(x),\dots,f_n(x))^T\in\mathbb{C}^{n\times 1},
\end{equation*}
 where $f_j(x)=\langle f(x),e_{j}\rangle_{E_0}\in\mathbb{C}$, for each $x\in G$. Following \cite{DuvanHomo}, for $f\in L^1(G,E_0)$ we define the Fourier coefficients of $f$ at $[\xi]\in \widehat{G}$ by 
 \begin{equation*}
     \widehat{f}(\xi)=(\widehat{f}_1(\xi),\dots,\widehat{f}_n(\xi))^T,
 \end{equation*}
 where  $[\xi]\in\widehat{G}$. If $A:C^\infty(G,E_0)\to C^\infty(G,E_0)$ is a continuous linear operator, then by defining its matrix-valued symbol 
\begin{equation}\label{symbol}
    \sigma_A(i,r,x,\xi) = \xi(x)^*e_r^*[A(\xi\otimes e_i)(x)],
\end{equation}
for $1\leq i,r\leq n$, $x\in G,[\xi]\in\widehat{G}$, it was proved in \cite{DuvanHomo} that the quantisation formula
\begin{equation}\label{quantisation1}
    Af(x) = \sum_{i,r=1}^n\sum_{[\xi]\in\widehat{G}}d_\xi\Tr[\xi(x)\sigma_A(i,r,x,[\xi])\widehat{f}_i(\xi)]e_r,
\end{equation}
holds true for every $x\in G$, $f\in C^\infty(G,E_0)$. Equivalently, one can also write it as 
\begin{equation*}
      Af(x) = \left(\sum_{i=1}^n\sum_{[\xi]\in\widehat{G}}d_\xi\Tr[\xi(x)\sigma_A(i,1,x,[\xi])\widehat{f}_i(\xi)],\dots,\sum_{i=1}^n\sum_{[\xi]\in\widehat{G}}d_\xi\Tr[\xi(x)\sigma_A(i,n,x,[\xi])\widehat{f}_i(\xi)]\right)^T,
\end{equation*}
or even,
\begin{equation*}
    Af(x) = \sum_{[\xi]\in\widehat{G}}d_\xi\Tr[\xi(x)^{\otimes n}\sigma_A(x,[\xi])\widehat{f}(\xi)],
\end{equation*}
where $\sigma_A(x,[\xi])\in (\mathbb{C}^{d_\xi\times d_\xi})^{n\times n}$ is given by:
\begin{equation*}
    \sigma_A(x,[\xi])_{ir}=\sigma_A(i,r,x,[\xi]),
\end{equation*}
for $1\leq i,r\leq n$, $(x,[\xi])\in G\times \widehat{G}$, and $\xi(x)^{\otimes n}$ is the block diagonal $(\mathbb{C}^{d_\xi\times d_\xi})^{n\times n}$-matrix given by
\begin{equation}
    \xi(x)^{\otimes n}=\text{diag}(\xi(x),\dots,\xi(x)),
\end{equation}
for all $(x,[\xi])\in G\times\widehat{G}$. Finally, for a matrix vector $v$ the trace above should be understood component-wise, i.e.: $(\Tr(v))_{j}=\Tr(v_{j})$, for $1\leq j\leq n$.\\
As before, we introduce the symbol classes with respect to these quantisation formulas.
\begin{definition}
    Let $J_n\doteq\{1,\dots,n\}$, $0\leq \rho,\delta\leq 1$. A symbol $\sigma\in\mathcal{S}_{\rho,\delta}^m(G\otimes\text{End}(E_0))$ is a mapping from $J_n\times J_n\times G\times \text{Rep}(G)$, smooth in $x$, such that, for any $\xi\in\text{Rep}(G)$ and $i,r\in J_n,\, \sigma(i,r,x,\xi)\in \mathbb{C}^{d_\xi\times d_\xi}$, and for any admissible collection of difference operators $\Delta_\xi^\alpha$, it satisfies
    \begin{equation*}
\|\Delta_\xi^\alpha\partial_x^\beta \sigma(i,r,x,\xi)\|_{op}\leq C_{\alpha\beta}\pseudoxi^{m-\rho|\alpha|+\delta|\beta|},
    \end{equation*}
    for some $C_{\alpha,\beta}>0$, for all multi-indices $\alpha,\beta$, and all $(i,r,x,[\xi])\in J_n\times J_n\times G\times \widehat{G}$.\\
    For a symbol $\sigma$, its associated operator $\text{Op}(\sigma):C^\infty(G,E_0)\to \mathcal{D}'(G,E_0)$ is defined by
    \begin{equation}\label{defAmpli1}
        (\text{Op}(\sigma)u)_r(x)\doteq
        \sum_{[\eta]\in\widehat{G}}d_\eta\Tr\left[\eta(x)\sum_{i=1}^n\sigma(i,r,x,\eta)\widehat{u_i}(\eta)\right],
    \end{equation}
    for every $1\leq r\leq n$. We then say $\text{Op}(\sigma)$ is a vector-valued pseudo-differential operator of order $m$, and write $\text{Op}(\sigma)\in\Psi^{m}_{\rho,\delta}(G\otimes\text{End}(E_0))$. 
\end{definition}
\noindent The vector-valued Sobolev space $H^s(G,E_0)$, $s\in \R$, is defined, as the completion of $C^\infty(G,E_0)$ under the Sobolev norm
\begin{equation*}
    \|u\|_{H^s(G,E_0)}^2=\sum_{[\xi]\in\widehat{G}}d_\xi\langle\xi\rangle^{2s}\sum_{i=1}^{d_\tau}\|\widehat{u}(i,\xi)\|_{HS}^2.
\end{equation*}The following theorem will be useful and its proof can be found in \cite{DuvanHomo}.

\begin{theorem}\label{pseudobounded}
    Let $A:C^\infty(G,E_0)\to C^\infty(G,E_0)$ be a continuous linear operator with symbol $a\in\mathcal{S}_{\rho,\delta}^m(G\otimes\text{End}(E_0))$, $0\leq \delta<\rho\leq 1$. Then $A:H^{s}(G,E_0)\to H^{s-m}(G,E_0)$ extends to a bounded operator for all $s\in\R$.
\end{theorem}

Next, we generalize the concept of symbols to amplitudes, as follows.

\begin{definition}
     An amplitude $a\in\mathcal{A}_{\rho,\delta}^m(G\otimes\text{End}(E_0))$, where $m\in\R$, $0\leq \rho,\delta\leq 1$,  is a mapping from $J_n\times J_n\times G\times G\times\text{Rep}(G)$, smooth in $x$ and $y$, such that, for any $\xi\in\text{Rep}(G),\, a(i,r,x,y,\xi)\in \mathbb{C}^{d_\xi\times d_\xi}$, and for any admissible collection of difference operators $\Delta_\xi^\alpha$, it satisfies
    \begin{equation*}
\|\Delta_\xi^\alpha\partial_x^\beta\partial_y^\gamma a(i,r,x,y,\xi)\|_{op}\leq C_{\alpha\beta\gamma}\pseudoxi^{m-\rho|\alpha|+\delta|\beta+\gamma|},
    \end{equation*}
    for all multi-indices $\alpha,\beta,\gamma$, $(i,r,x,y,[\xi])\in J_n\times J_n\times G\times G\times \widehat{G}$.\\
    For an amplitude $a$, the amplitude operator $\text{Op}(a):C^\infty(G,E_0)\to \mathcal{D}'(G,E_0)$ is defined by
    \begin{equation}\label{defAmpli}
        (\text{Op(a)}u)_r(x)\doteq
        \sum_{[\eta]\in\widehat{G}}d_\eta\Tr\left[\eta(x)\int_G\sum_{i=1}^na(i,r,x,y,\eta)u_i(y)\eta^*(y)\mathop{dy}\right]
    \end{equation}
    for every $1\leq r\leq n$. Notice that if $A$ is a pseudo-differential operator $A:C^\infty(G,E_0)\to C^\infty(G,E_0)$ defined as before, and if $a(i,r,x,y,\eta)=\sigma_A(i,r,x,\eta)$, then $\text{Op}(a)=A$.
\end{definition}

\begin{prop}\label{Propamplitude}
    Let $0\leq \delta<1$ and $0\leq \rho \leq 1$, and $a\in \mathcal{A}_{\rho,\delta}^m(G\otimes\text{End}(E_0))$. Then $\text{Op}(a)$ is a continuous linear operator from $C^\infty(G,E_0)$ to $C^\infty(G,E_0)$. 
\end{prop}
\begin{proof}
    Recall that by definition, $(1+\mathcal{L}_G)\eta=\langle\eta\rangle^2 \eta$. Consequently, integrating by parts in the $dy$-integral in \eqref{defAmpli} with respect to $\langle\eta\rangle^{-2}(1+\mathcal{L}_G)$ arbitrarily many times, we obtain that the $\eta$-series in \eqref{defAmpli} converges uniformly and absolutely, so that $\text{Op}(a)\in C^\infty(G,E_0)$ provided $u\in C^\infty(G,E_0)$. The continuity of  $\text{Op}(a)$ follows similarly.
\end{proof}

\begin{prop}\label{propasymptotic}
    Let $m\in\R$, $0\leq \delta<\rho\leq 1$ and $a\in\mathcal{A}_{\rho,\delta}^m(G\otimes\text{End}(E_0))$. Then $A=\text{Op}(a)$ is a pseudo-differential operator with matrix symbol $\sigma_A\in\mathcal{S}_{\rho,\delta}^m(G\otimes\text{End}(E_0))$. Moreover, $\sigma_A$ has the asymptotic expansion
    \begin{equation*}
        \sigma_A(i,r,x,\xi)\sim\sum_{\alpha\geq 0}\frac{1}{\alpha!}\partial_y^\alpha\Delta_\xi^\alpha a(i,r,x,y,\xi)|_{y=x}
    \end{equation*}
    for each $1\leq i,r,\leq n$, in the sense that 
    \begin{equation*}
        \left(\sigma_A(i,r,x,\xi)-\sum_{0\leq|\alpha|<N}\frac{1}{\alpha!}\partial_y^\alpha\Delta_\xi^\alpha a(i,r,x,y,\xi)|_{y=x}\right)\in \mathcal{S}_{\rho,\delta}^{m-(\rho-\delta)N}(G\otimes\text{End}(E_0)),
    \end{equation*}
    for every $N\in\N$ sufficiently big.
\end{prop}
\begin{proof}
    Indeed, by Proposition \ref{Propamplitude}, $A$ is a continuous linear operator acting on $C^\infty(G,E_0)$ and therefore admits symbol $\sigma_A$, the matrix symbol of $A$, given by:
    \begin{equation*}
        \sigma_A(i,r,x,\xi)=\xi(x)^*e_r^*[A(\xi\otimes e_i)(x)].
    \end{equation*}
    Therefore
    \allowdisplaybreaks\begin{align*}
\sigma_A(i,r,x,\xi)_{mn}&=\sum_{l=1}^{d_\xi}\xi(x^{-1})_{ml}(A(\xi_{ln}\otimes e_i))_r(x)\notag\\
        &= \sum_{l=1}^{d_\xi}\xi(x^{-1})_{ml}\int_G\sum_{[\eta]\in\widehat{G}}d_\eta \Tr[\eta(x)a(i,r,x,y,\eta)\xi(y)_{ln}\eta(y)^*]dy\notag\\
        &=\int_G\sum_{l=1}^{d_\xi}\xi(x^{-1})_{ml}\xi(y)_{ln}\sum_{[\eta]\in\widehat{G}}d_\eta \Tr[\eta(y^{-1})\eta(x)a(i,r,x,y,\eta)]dy\notag\\
        &=\int_G\xi(x^{-1}y)_{mn}\sum_{[\eta]\in\widehat{G}}d_\eta \Tr[\eta(y^{-1}x)a(i,r,x,y,\eta)]dy.\notag\\
        \end{align*}
        Performing the change of variables $z=y^{-1}x$ (justified by the invariance of the Haar measure), we obtain that
        \begin{align}   
        \sigma_A(i,r,x,\xi)_{mn}
        &=\int_G\xi(z^{-1})_{mn}\sum_{[\eta]\in\widehat{G}}d_\eta\sum_{j,k=1}^{d_\eta}\eta(z)_{jk}\,a(i,r,x,xz^{-1},\eta)_{kj}dz\notag\\
        &= \sum_{0\leq|\alpha|<N}\frac{1}{\alpha!}\partial_y^{\alpha}|_{y=x}\int_G\xi(z^{-1})_{mn}q_{\alpha}(z)\sum_{[\eta]\in\widehat{G}}d_\eta\sum_{j,k=1}^{d_\eta}\eta(z)_{jk}a(i,r,x,y,\eta)_{kj}dz\notag\\
        &+
        \int_G\xi(z^{-1})_{mn}\sum_{[\eta]\in\widehat{G}}d_\eta\sum_{j,k=1}^{d_\eta}\eta(z)_{jk}R_N(i,r,z,\eta)_{kj}dz\notag\\
        &=\sum_{0\leq|\alpha|<N}\frac{1}{\alpha!}\int_G\xi(z)_{mn}^*q_{\alpha}(z)\sum_{[\eta]\in\widehat{G}}d_\eta\Tr[\eta(z)\partial_y^{\alpha}|_{y=x}a(i,r,x,y,\eta)]dz\notag\\
        &+\int_G\xi(z^{-1})_{mn}\sum_{[\eta]\in\widehat{G}}d_\eta\sum_{j,k=1}^{d_\eta}\eta(z)_{jk}R_N(i,r,z,\eta)_{kj}dz\notag\\
        &=\sum_{0\leq|\alpha|<N}\frac{1}{\alpha!}\partial_y^\alpha\Delta_\xi^\alpha a(i,r,x,y,\xi)_{mn}|_{y=x}\notag\\
        &+\sum_{[\eta]\in\widehat{G}}d_\eta
    \sum_{j,k=1}^n\int_G\xi(z^{-1})_{mn}\eta(z)_{jk}R_N(i,r,z,\eta)_{kj}dz,\label{remainder}
   \end{align}
   where $R_N(i,r,z,\eta)_{kj}$ is the remainder of the Taylor expansion in $z$ of order $N\in\N$ of $a(i,r,x,z,\eta)_{kj}$ centered at $x$. Therefore, to prove the claim we only need to analyze these last series of integrals. For $(x,y)\in G\times G$, $1\leq i, r\leq n$, define $k_{A,x,y}(i,r,\cdot)$ to be the distribution defined by
   \begin{equation*}
       a(i,r,x,y,\xi)=\widehat{k_{A,x,y}}(i,r,\xi),\,[\xi]\in\widehat{G}.
   \end{equation*}
   Similarly, denote by $k_{\sigma,x}(i,r,\cdot)$ to be the right convolution kernels of $A$, that is, the distributions which satisfy
   \begin{equation*}
       \sigma_A(i,r,x,\xi)=\widehat{k_{\sigma,x}}(i,r,\xi),\,[\xi]\in\widehat{G}.
   \end{equation*}
Since
\begin{align*}
    Af(x)_r&=\int_G \sum_{i=1}^{n}f_i(y)k_{A,x,y}(i,r,y^{-1}x)dy\\
    &=\int_G\sum_{i=1}^{n}f(xz^{-1})k_{A,x,xz^{-1}}(z)dz,
\end{align*}
and
\begin{equation*}
    Af(x)_r=\int_G \sum_{i=1}^{n}f_i(y)k_{\sigma,x}(i,r,y^{-1}x)= \sum_{i=1}^{n}f_i(xz^{-1})k_{\sigma,x}(i,r,z),
\end{equation*}
we have that
\begin{equation*}
    k_{A,x,y}(i,r,y^{-1}x)=k_{\sigma,x}(i,r,y^{-1}x),\,k_{A,x,xz^{-1}}(i,r,z)=k_{\sigma,x}(i,r,z),
\end{equation*}
for all $1\leq i\leq r,\leq n$, $x,y,z\in G$, in the sense of distributions on $G$. Now, fix $1\leq i,r\leq n$, and notice that the last term in \eqref{remainder} is the inverse Fourier transform of
\begin{equation*}
    k_{\sigma,x}(i,r,z)-\sum_{|\alpha|<N}q_{\alpha}(z)\partial_z^{\alpha}k_{A,x,xz_1^{-1}}(i,r,z)|_{z_{1}=e}.
\end{equation*}
Our goal now is to prove that the Fourier transform of the expression above is in $\mathcal{S}_{\rho,\delta}^{m-(\rho-\delta)N}(G\otimes \text{End}(E_0))$ for every $N$ sufficiently big, which we do as follows. For any multi-indices $\gamma,\beta$, let $M'\in2\N_0$ such that
\begin{equation*}
    M'>\rho |\gamma|-\delta|\beta|-m+(\rho-\delta)N,
\end{equation*}
and
\begin{equation}\label{ineqM'}
    \frac{M'}{2}>-\rho|\gamma|+\delta|\beta|+m+2(\rho-\delta)N.
\end{equation}
Then 
\begin{align}
    &\left\|\langle\xi\rangle^{\rho|\gamma|-\delta|\beta|-m+(\rho-\delta)N} \Delta^{\gamma}\partial_x^{\beta}\left[\sum_{[\eta]\in\widehat{G}}d_\eta\sum_{i=1}^n\sum_{j,k=1}^n\int_G\xi(z^{-1})_{mn}\eta(z)_{jk}R_N(i,r,z,\eta)_{kj}dz\right]\right\|_{op}\notag\\
     &\leq\left\|\langle\xi\rangle^{M'} \Delta^{\gamma}\partial_x^{\beta}\left[\sum_{[\eta]\in\widehat{G}}d_\eta\sum_{i=1}^n\sum_{j,k=1}^n\int_G\xi(z^{-1})_{mn}\eta(z)_{jk}R_N(i,r,z,\eta)_{kj}dz\right]\right\|_{op}\notag\\
    &\leq \|(\text{Id}+\mathcal{L})^{\frac{M'}{2}}[q_\gamma(z)\partial_x^\beta[k_{\sigma,x}(i,r,z)-\sum_{|\alpha|<N}q_{\alpha}(z)\partial_z^{\alpha}k_{A,x,xz_1^{-1}}(i,r,z)|_{z_{1}=e}]\|_{L^1(G)_z}\notag\\
    &= \|(\text{Id}+\mathcal{L})^{\frac{M'}{2}}[q_\gamma(z)\partial_x^\beta R_{x,N}^{k_{A,x,x}}(i,r,z)]\|_{L^1(G)_z},\label{Rem}
\end{align}
where
\begin{equation*}
    k_{\sigma,x}(i,r,z)=k_{A,x,xz^{-1}}(i,r,z)=\sum_{|\alpha|<N}q_{\alpha}(z)\partial_{z_1}^{\alpha}k_{A,x,xz_1^{-1}}(i,r,z)|_{z_{1}=e}+R_{x,N}^{k_{A,x}}(i,r,z)
\end{equation*}
for all $x,y,z\in G$ by the Taylor expansion theorem. Moreover, by \cite{VerIntri}, Lemma 7.4, the remainder satisfies the estimate
\begin{equation*}
    |R_{x,N}(i,r,z)|\leq C|z|^N\max_{|\alpha|\leq N}\|\partial^\alpha_x k_{A,x,x}(i,r,z)\|_{L^\infty(G)_x}.
\end{equation*} Now applying the definition of $\mathcal{L}$ and the Leibniz's rule on \eqref{Rem}, gives us
\begin{align*}
    \|(\text{Id}+\mathcal{L})^{\frac{M'}{2}}[q_\gamma(z)\partial_x^\beta &R_{x,N}^{k_{A,x,x}}(i,r,z)]\|_{L^1(G)_z}\\
    &\lesssim \sum_{1\leq i_1\leq\dots\leq i_d\leq d,\,|\lambda|\leq M'}\|X_{i_{1},z}^{\lambda_1}\dots X_{i_{d},z}^{\lambda_d} [R_{x,N}^{q_\gamma(z)\partial_x^{\beta}k_{A,x,x}}(i,r,z)\|_{L^1(G)_z}.
\end{align*}
By estimates similar to the ones above, we have
\begin{equation*}
    |X_{i_{1},z}^{\lambda_1}\dots X_{i_{d},z}^{\lambda_d} [R_{x,N}^{q_\gamma(z)\partial_x^{\beta}k_{A,x,x}}(i,r,z)]|\lesssim |z|^{N-|\lambda|}\max_{|\alpha|\leq N-|\lambda|}\|\partial_z^{\alpha+\lambda}(q_\gamma(z)\partial_x^{\beta}k_{A,x,x}(i,r,z))\|_{L^\infty(G)_x},
\end{equation*}
and since $\{\partial_z^{\alpha+\lambda}(q_\gamma(\cdot)\partial_x^{\beta}k_{A,x,x}(i,r,\cdot))\}_{i,r=1}^n$ are the right-convolution kernels of a pseudo-differential operator of order 
\begin{equation*}
    s'=m+\delta|\beta|+\delta|\alpha|+\delta|\lambda|-\rho|\gamma|,
\end{equation*}
by Proposition 6.7 of \cite{VerIntri} we have that
\begin{align*}
    \|\partial_z^{\alpha+\lambda}(q_\gamma(z)\partial_x^{\beta}k_{A,x,x}(i,r,z))\|_{L^\infty(G)_z}\lesssim |z|^{-\frac{s'+d}{\rho}},
\end{align*}
for $1\leq i,r\leq n$. Putting these previous estimates together, one obtains
\begin{equation}\label{integral}
    \|(\text{Id}+\mathcal{L})^{\frac{M'}{2}}[q_\gamma(z)\partial_x^\beta R_{x,N}^{k_{A,x,x}}(i,r,z)]\|_{L^1(G)_z}\lesssim \int_G |z|^{N-|\lambda|}|z|^{-\frac{s'+d}{\rho}},
\end{equation}
and the integral on the right-hand-side is finite if 
\begin{align*}
    \rho(|\lambda|-N)+s'+d< \rho d.
\end{align*}
But indeed, for $N$ big enough, we have $ M'\lesssim (\rho-\delta)N$. By choosing $M'$ so that $(\rho-\delta)N/2\lesssim M'$ also, we have that $(\rho-\delta)N\asymp M'$. Hence, using that $|\alpha|+|\lambda|\leq N$ and \eqref{ineqM'} yields
\begin{align*}
    \rho(|\lambda|-N)+s'+d&=\rho|\lambda|-\rho d+m+\delta|\beta|+\delta|\alpha|+\delta|\lambda|-\rho|\gamma|+d\\
    &< \rho|\lambda|-\rho N+m+\delta|\beta|+\rho|\alpha|+\rho|\lambda|-\rho|\gamma|+n\\
    &=\rho|\lambda|-\rho(N-|\alpha|-|\lambda|)+m+\delta|\beta|-\rho|\gamma|+d\\
    &\leq \rho|\lambda|+m+\delta|\beta|-\rho|\gamma|+d\\
    &\leq \rho M'+\frac{M'}{2}-2N(\rho-\delta)+d\leq \frac{3M'}{2}-2N(\rho-\delta)+d\\
    &\asymp \frac{3(\rho-\delta)N}{2}-2N(\rho-\delta)+d=-\frac{(\rho-\delta)N}{2}+d.
\end{align*}
By choosing $N\geq N_0\geq 2d(\rho-1)/(\rho-\delta)$, we get that $|\lambda|-N+s'+d<\rho d$, which by \eqref{integral} and the previous inequalities proves that the remainder term in the asymptotic expansion is in $\mathcal{S}_{\rho,\delta}^{m-(\rho-\delta)N}(G\otimes \text{End}(E_0))$, for all such $N$.
\end{proof}

\subsection{Even and odd functions on compact Lie groups}

Recall that $f\in C^{\infty}(G)$ is called central if $f(xy)=f(yx)$, for every $x,y\in G$. Following \cite{RuzSharp}, we will say $f\in C^\infty(G)$ is {\it even} if it is invariant under inversions, that is: $f(x^{-1})=f(x)$ for every $x\in G$. Similarly, we will say that $f\in C^\infty(G)$ is {\it odd} if $f(x^{-1})=-f(x)$, for every $x\in G$. The following results are then easy to verify:

\begin{lemma}\label{evenodd}
    Let $f\in C^\infty(G)$ be central and even. Then for any left-invariant vector field $X$, the function $Xf$ is odd. Moreover, if $g\in C^{\infty}$ is odd, then $fg$ is also odd.
\end{lemma}

\begin{lemma}\label{int0}
    Let $f\in C^\infty(G)$ be odd. Then $\int_G f(x)dx=0$.
\end{lemma}

\section{Main Results}\label{MainRes}

\subsection{Sharp G\aa rding inequality on compact Lie groups}
Here we present the main result of this paper. Its proof will be then carried out through several lemmas.

\begin{theorem}\label{maintheo}
    Let $G$ be a compact Lie group, $\dim(G)=d$, $E_0$ an finite a $n$-dimensional $\mathbb{C}$-vector space. Let $A=\text{Op}(\sigma_A)\in\Psi^m_{\rho,\delta}(G\otimes \text{End}(E_0))$, $0\leq \delta<\rho\leq 1$, be such that its matrix-valued symbol $\sigma_A(x,\xi)$ is positive semi-definite for every $(x,[\xi])\in G\times \widehat{G}$, in the sense that $B^*\sigma_A(x,\xi)B\in \mathbb{C}^{d_\xi\times d_\xi}$ is positive semi-definite for all $B\in(\mathbb{C}^{d_\xi\times d_\xi})^{n\times 1}$. Equivalently, this means that for any $B_1,\dots,B_n\in \mathbb{C}^{d_\xi\times 1}$, the inequality
    \begin{align*}
        \sum_{i,r=1}^n \overline{B_i}^T\sigma_A(i,r,x,\xi)B_r\geq 0,
    \end{align*}
    holds for all $(x,[\xi])\in G\times\widehat{G}$.
    Then there exists $C>0$ such that, for every $u\in C^\infty(G,E_0)$ we have
    \begin{equation}\label{ineqtheo}
        \text{Re}(Au,u)_{L^2(G,E_0)}\geq -C\|u\|^2_{H^{(m-(\rho-\delta))/2}(G,E_0)}.
    \end{equation}
\end{theorem}
\begin{remark}\label{remark}
    In the previous statement and throughout the rest of this paper, for a column-vector of matrices $B\in(\mathbb{C}^{d_\xi\times d_\xi})^{n\times 1}$ we denote by $B^*\in(\mathbb{C}^{d_\xi\times d_\xi})^{1\times n}$ the row-vector of matrices given by
    \begin{equation*}
        B^*=\left(B_{1}^*,\dots,B_{n}^*\right).
    \end{equation*}
\end{remark}
    First, notice that if $Q:H^{\frac{m-(\rho-\delta)}{2}}(G,E_0)\to H^{-\frac{m-(\rho-\delta)}{2}}(G,E_0)$ is a bounded linear operator, we have
    \begin{align*}
        \text{Re}(Qu,u)_{L^2(G,E_0)}&\geq -|(Qu,u)_{L^2(G,E_0)}|\\
        &\geq -\|Qu\|_{H^{-\frac{m-(\rho-\delta)}{2}}(G,E_0)}\|u\|_{H^{\frac{m-(\rho-\delta)}{2}}(G,E_0)}\\
        &\geq -\|Q\|_{op}\|u\|_{H^{\frac{m-(\rho-\delta)}{2}}(G,E_0)}^2.
    \end{align*}
Hence, Theorem \ref{maintheo} will follow once we show that $A$ can be written as $A=P+Q$, where $P$ is positive and $Q$ is as above.\\
First, following the ideas in \cite{RuzSharp} and \cite{SharpSub}, we construct an auxiliary function $w_\xi:G\to \mathbb{C}$, for each $\xi\in\text{Rep}(G)$.\\
We can assume $G$ is a closed subgroup of $\text{GL}(N,\mathbb{R})$ for some $N\in\N$. Then its Lie algebra $\mathfrak{g}\subset\R^{N\times N}$ is a $d$-dimensional vector subspace such that $[A,B]\doteq AB-BA\in\mathfrak{g}$, for every $A,B\in\mathfrak{g}$. Let $U\subset G$ and $V\subset \mathfrak{g}$ be neighbourhoods of the identity $\text{Id}=e\in G$ and $0\in\mathfrak{g}$, respectively, so that the matrix exponential mapping is a diffeomorphism $\exp:V\to U$. Without loss of generality, we may assume that $V$ is the open ball $V= B(0,r)=\{z\in\R^d|\,|z|<r\}$, of radius $r>0$. Let $\phi:[0,\infty)\to [0,\infty)$ be a smooth function such that the mapping
\begin{align*}
    \mathfrak{g}&\to\R\\
    z&\mapsto \phi(|z|)
\end{align*}
is supported in $V$ and such that $\phi(s)=1$ for all sufficiently small $s>0$. For every $\xi\in \text{Rep}(G)$ define
    \begin{align*}\label{w_xi}
       w_\xi:G&\to\R\\ 
        x&\mapsto \phi(|\exp^{-1}(x)|\langle\xi\rangle^{\frac{\rho+\delta}{2}})\psi(\exp^{-1}(x))\langle\xi\rangle^{\frac{d(\rho+\delta)}{4}},
    \end{align*}
    where $\psi(y) \doteq C_0|\det D\exp(y)|^{-\frac{1}{2}}f(y)^{-\frac{1}{2}}$, for every $y\in\mathfrak{g}\cong\R^d$, $D\exp$ is the Jacobi matrix of the mapping $\exp$, $f$ is the density with respect to the Lebesgue measure of the Haar measure on $G$ pulled back to $\mathfrak{g}\cong\R^d$ by the exponential mapping, and with $C_0=(\int_{\R^d}\phi(|z|)dz)^{-\frac{1}{2}}$.
    
    The following lemma states the main properties of $w_\xi$ that will be used in this paper, and its proof can be found in \cite{SharpSub}.

\begin{lemma}\label{lemmaw_xi}
    The functions $w_\xi$ defined above are smooth, for every $\xi\in \text{Rep}(G)$. Moreover, they satisfy the following properties:
    \begin{enumerate}
        \item $w_\xi(e)=C_0\langle\xi\rangle^{\frac{d(\rho+\delta)}{4}}$;
        \item $w_\xi$ is central and inversion invariant;
        \item $\text{dist}(x,e)\sim |\exp^{-1}(x)|\lesssim \pseudoxi^{-\frac{\rho+\delta}{2}}$ on the support of $w_\xi$;
        \item $\|w_\xi\|_{L^2(G)}=1$;
        \item $(x,[\xi])\mapsto w_\xi(x)\text{Id}_{d_\xi}\in \mathcal{S}^{d(\rho+\delta)/4}_{\rho,(\rho+\delta)/2}(G)$,
    \end{enumerate}
 for every $\xi\in\text{Rep}(G)$.
\end{lemma}

\begin{lemma}\label{lemmaP}
    Let $\sigma_A$ be the matrix valued symbol of $A$ as in Theorem \ref{maintheo}. Define the amplitude
    \begin{equation}\label{P}
        p(i,r,x,y,\xi)\doteq \int_Gw_{\xi}(xz^{-1})w_\xi(yz^{-1})\sigma_A(i,r,z,\xi)dz\in (\mathbb{C}^{d_\xi\times d_\xi})^{n\times n},
    \end{equation}
    for every $1\leq i,r\leq n,\, x,y\in G$, $[\xi]\in\widehat{G}$, where $w_\xi$ is as in Lemma \ref{lemmaw_xi}. Then $p\in\mathcal{A}_{\rho,\frac{\rho+\delta}{2}}^{m}(G\otimes \text{End}(E_0))$ and the linear operator $P:C^\infty(G,E_0)\to C^\infty(G,E_0)$ given by
    \begin{equation*}
(Pu)_r(x)=\int_G\sum_{i=1}^n\sum_{[\xi]\in\widehat{G}}d_\xi\Tr(\xi(y^{-1}x)p(i,r,x,y,\xi)u_i(y))dy,\quad 1\leq r\leq n,
    \end{equation*}
    for every $x\in G$ is a well defined pseudo-differential operator of order $m$, and also a positive linear operator on $L^2(G,E_0)$.
\end{lemma}
\begin{proof} First, we verify that $p\in\mathcal{A}_{\rho,\frac{\rho+\delta}{2}}^m(G\otimes \text{End}(E_0))$. Indeed, note that by the Leibniz's rule, for any multi-indices $\alpha,
\beta,\gamma$, the matrix $\partial_x^{\beta}\partial_y^{\gamma}\Delta_\xi^{\alpha}p(i,r,x,y,\xi)$ is given by a sum of terms of the form
\begin{equation*}
    \int_G(\Delta_\xi^{\eta}\partial_x^{\beta}w_\xi(xz^{-1}))(\Delta_\xi^{\lambda}\partial_y^{\gamma}w_{\xi}(yz^{-1}))(\Delta_\xi^{\mu}\sigma_A(i,r,z,\xi))dz
\end{equation*}
where $|\eta+\lambda+\mu|\geq |\alpha|$. Also, due to properties of $w_\xi$ from Lemma \ref{lemmaw_xi}, we have
\begin{align*}
    \|(\Delta_\xi^{\eta}\partial_x^{\beta}w_\xi(xz^{-1}))(\Delta_\xi^{\lambda}\partial_y^{\gamma}w_{\xi}(yz^{-1}))\|_{op}&\lesssim \pseudoxi^{d\frac{(\rho+\delta)}{2}-\rho(|\eta|+|\lambda|)+\frac{(\rho+\delta)}{2}(|\beta|+|\gamma|)}.
\end{align*}
Therefore, since $\text{supp}(w_\xi)$ is contained in a set of measure proportional to $\pseudoxi^{-d\frac{(\rho+\delta)}{2}}$, we can estimate
\begin{align*}
    &\left\|\int_G(\Delta_\xi^{\eta}\partial_x^{\beta}w_\xi(xz^{-1}))(\Delta_\xi^{\lambda}\partial_y^{\gamma}w_{\xi}(yz^{-1}))(\Delta_\xi^{\mu}\sigma_A(i,r,z,\xi))dz\right\|_{op}\\
    &\qquad\qquad\qquad\qquad\qquad\qquad\qquad\lesssim \pseudoxi^{d\frac{(\rho+\delta)}{2}-\rho(|\eta|+|\lambda|)+\frac{(\rho+\delta)}{2}(|\beta|+|\gamma|)-d\frac{(\rho+\delta)}{2}+m-\rho|\mu|}\\
    &\qquad\qquad\qquad\qquad\qquad\qquad\qquad=\pseudoxi^{m-\rho|\alpha|+\frac{(\rho+\delta)}{2}(|\beta|+|\gamma|)}
\end{align*}
and so $\|\partial_x^{\beta}\partial_y^{\gamma}\Delta_\xi^{\alpha}p(i,r,x,y,\xi)\|_{op}\lesssim\pseudoxi^{m-\rho|\alpha|+\frac{(\rho+\delta)}{2}(|\beta|+|\gamma|)}$ also, which proves the first claim.
Next, notice that
    \begin{align*}
        (Pu,u)_{L^2(G,E_0)}&=\int_G\langle Pu(x),u(x)\rangle_{E_0}dx\\
        &=\int_G \int_G\sum_{r,i=1}^n\sum_{[\xi]\in\widehat{G}}d_\xi\Tr[\xi(y)^*\xi(x)p(i,r,x,y,\xi)u_i(y)]\mathop{dy}\overline{u_r(x)}\mathop{dx}.
    \end{align*}
    Substituting $p(i,r,x,y,\xi)$ by its definition in the previous expression, one obtains
    \begin{align}\label{eqfeia2}
    \int_G\sum_{[\xi]\in\widehat{G}}\sum_{r,i=1}^n\Tr[\xi(x)\int_Gw_\xi(xz^{-1})w_\xi(yz^{-1})\sigma_A(i,r,z,\xi)\mathop{dz}u_i(y)\xi(y)^*]\mathop{dy}\overline{u_r(x)}\mathop{dx}.
    \end{align}
    Let 
    \begin{equation*}
        M(i,z,\xi)\doteq \int_Gw_{\xi}(yz^{-1})\xi(yz^{-1})^*u_i(y)dy\in \mathbb{C}^{d_\xi\times d_\xi}.
    \end{equation*}
      Then \eqref{eqfeia2} can be written as
    \begin{align*}
        \int_G\sum_{i,r=1}^n\sum_{[\xi]\in\widehat{G}}d_\xi \Tr[M(i,z,\xi)^*\sigma_A(i,r,z,\xi)M(r,z,\xi)]dz,
    \end{align*}
    which is non-negative by our hypothesis since
    \begin{align*}
    \Tr[M(i,z,\xi)^*\sigma_A(i,r,z,\xi)M(r,z,\xi)]&=\sum_{k=1}^{d_\xi}f_k^*M(i,z,\xi)^*\sigma_A(i,r,z,\xi)M(r,z,\xi)f_k\geq 0,
    \end{align*}
    as $M(r,z,\xi)f_k\in\mathbb{C}^{d_\xi\times 1}$, where $\{f_k\}_{k=1}^{d_\xi}$ is any orthonormal basis of column vectors in $\mathbb{C}^{d_\xi}$. 
\end{proof}

\begin{lemma}\label{lemma1}
    Following the previous notation, we have that $(i,r,x,[\xi])\mapsto p(i,r,x,x,\xi)-\sigma_A(x,\xi)$ is the symbol of a pseudo-differential  operator bounded from $H^s(G,E_0)$ to \\$H^{s-(m-(\rho-\delta))}(G,E_0)$, for any $s\in\R$.
\end{lemma}
\begin{proof}
    By Theorem \ref{pseudobounded}, it is enough to show that 
    \begin{equation*}
        (i,r,x,[\xi])\mapsto p(i,r,x,x,\xi)-\sigma_A(i,r,x,\xi)\in\mathcal{S}^{m-(\rho-\delta)}_{\rho,\delta}(G\otimes\text{End}(E_0)).
    \end{equation*}
    Notice that, by Lemma \ref{lemmaw_xi}, 
    \begin{align*}
        p(i,r,x,x,\xi)-\sigma_A(i,r,x,\xi)&=\int_Gw_\xi(z)^2\sigma_A(i,r,xz^{-1},\xi)dz-\sigma_A(i,r,x,\xi)\\
        &=\int_Gw_\xi(z)^2(\sigma_A(i,r,xz^{-1},\xi)- \sigma_A(i,r,x,\xi))dz.
    \end{align*}
    Using the Taylor expansion of $\sigma_A(i,r,xz^{-1},\xi)$ at $x$ we can write
    \begin{align*}
\sigma_A(i,r,xz^{-1},\xi)=\sigma_A(i,r,x,\xi)+\sum_{|\gamma|=1}\partial_x^{\gamma}\sigma_A(i,r,x,\xi)q_\gamma(z)+R_x(i,r,z,\xi),
    \end{align*}
    where $R_x \in\mathcal{S}_{\rho,\delta}^{m+2\delta}(G\otimes \text{End}(E_0))$ is the Taylor remainder of order $2$ of $\sigma_A(\cdot,\cdot,x,\cdot)$. Notice that we can choose the ``polynomials" $q_\gamma$ so that they are odd for all $|\gamma|=1$, and using that $w_\xi$ is even, we can conclude that
    \begin{equation*}
        \int_Gw_\xi^2(z)q_\gamma(z)dz=0,
    \end{equation*}
    for all $|\gamma|=1$. Hence, for multi-indices $\alpha,
    \beta$:
    \begin{align*}
        \Delta^\alpha_\xi\partial_x^\beta(p(i,r,x,x,\xi)-\sigma_A(i,r,x,\xi))= \Delta^\alpha_\xi\int_Gw_\xi(z)^2 \partial_x^\beta R_x(i,r,z,\xi)\mathop{dz}.
    \end{align*}
    Applying the Leibniz rule, we can write the expression above as a sum of terms of the form
    \begin{align*}
\int_G\partial_x^\beta\Delta_\xi^{\alpha_1}R_{x}(i,r,z,\xi)\Delta_\xi^{\alpha_2}w_\xi(z)\Delta_\xi^{\alpha_3}w_\xi(z)dz,
    \end{align*}
    where $|\alpha_1+\alpha_2+\alpha_3|\geq |\alpha|$.
By Lemma 7.4 of \cite{VerIntri}, the remainder satisfies 
\begin{equation*}    \|\Delta_\xi^{\alpha_1}\partial_x^{\beta}R_x(i,r,z,\xi)\|_{op}\leq C|z|^2\max_{|\gamma|\leq 2}\sup_{x\in G}\|\Delta_\xi^{\alpha_1}\partial_x^\gamma\partial_x^\beta\sigma_A(i,r,x,\xi)\|_{op}\lesssim |z|^2\pseudoxi^{m-\rho|\alpha_1|+\delta(2+|\beta|)},
\end{equation*}
for all $z\in G$, $|z|\lesssim \langle\xi\rangle^{-\frac{\rho+\delta}{2}}$ on $\text{supp}(w_\xi)$ by Lemma \ref{lemmaw_xi}, and $|\text{supp}(w_\xi)|\lesssim \langle\xi\rangle^{-d\frac{(\rho+\delta)}{2}}$,  we conclude that 
\begin{align*}
    &\left\|\int_G\partial_x^\beta\Delta_\xi^{\alpha_1}R_{x}(i,r,z,\xi)\Delta_\xi^{\alpha_2}w_\xi(z)\Delta_\xi^{\alpha_3}w_\xi(z)dz\right\|_{op}\\
    &\qquad\qquad\qquad\qquad\qquad\qquad\qquad\lesssim \pseudoxi^{-\rho-\delta+m-\rho|\alpha_1|+\delta(2+|\beta|)+d\frac{(\rho+\delta)}{4}-\rho|\alpha_2|+d\frac{(\rho+\delta)}{4}-\rho|\alpha_3|-d\frac{(\rho+\delta)}{2}}\\
    &\qquad\qquad\qquad\qquad\qquad\qquad\qquad=\pseudoxi^{m-(\rho-\delta)-\rho|\alpha|+\delta|\beta|}, 
\end{align*}
which implies 
\begin{equation*}
    \|\Delta^\alpha_\xi\partial_x^\beta(p(i,r,x,x,\xi)-\sigma_A(i,r,x,\xi))\|_{op}\lesssim \pseudoxi^{m-(\rho-\delta)-\rho|\alpha|+\delta|\beta|}
\end{equation*}
proving the claim.
\end{proof}
\begin{lemma}\label{lemma2}
    Let $\sigma_P$ be the symbol of the pseudo-differential operator $P$ defined in Lemma \ref{lemmaP}. Then the pseudo-differential operator with symbol $\{\sigma_P(i,r,x,\xi)-p\}_{i,r=1}^n$ is bounded from $H^s(G,E_0)$ to $H^{s-(m-(\rho-\delta))}(G,E_0)$, for any $s\in\R$.
\end{lemma}
\begin{proof}
    As in the proof of the previous lemma, it is enough to prove that
    \begin{equation*}
         (i,r,x,[\xi])\mapsto \sigma_P(i,r,x,\xi)-p(i,r,x,x,\xi)\in \mathcal{S}_{\rho,\frac{(\rho+\delta)}{2}}^{m-(\rho-\delta)}(G\otimes\text{End}(E_0)).
    \end{equation*}
  By Proposition \ref{propasymptotic}, we have the asymptotic expansion
 \begin{equation*}
     \sigma_P(i,r,x,\xi)\sim\sum_{\alpha\geq 0}\frac{1}{\alpha!}\Delta_\xi^\alpha\partial_y^\alpha p(i,r,x,y,\xi)|_{y=x},
 \end{equation*}
 with the properties specified in Proposition \ref{propasymptotic}.
 For now, recalling the formula \eqref{P}, notice that after a change of variables $ z^{-1}x\mapsto z$ one obtains
 \begin{equation*}
     \Delta_\xi^\alpha\partial_y^\alpha p(i,r,x,y,\xi)|_{y=x} = \Delta_\xi^\alpha\int_Gw_\xi(z)\partial_z^\alpha w_\xi(z)\sigma_A(i,r,z^{-1}x,\xi)\mathop{dz}.
 \end{equation*}
 Let $N\in \N$ to be chosen later. Define 
 \begin{equation*}
     S_N(i,r,x,\xi)=\sigma_P(i,r,x,\xi)-p(i,r,x,x)-R_N(i,r,x,\xi),
 \end{equation*}
 where $R_N\in\mathcal{S}_{\rho,\delta}^{m-\frac{(\rho+\delta)}{2}(N+1)}(G\otimes\text{End}(E_0))$ is the remainder term given by the asymptotic expansion above. Then applying the Taylor expansion to the formula above yields
 \begin{align*}\label{Leibnizugly}
     S_N(i,r,x,\xi)&=\sum_{1\leq |\alpha|\leq N}\Delta_\xi^\alpha\int_Gw_\xi(z)\partial_z^\alpha w_\xi(z)\sigma_A(i,r,x,\xi)\mathop{dz}\\
     &+\sum_{1\leq |\alpha|\leq N}\Delta_\xi^\alpha\int_Gw_\xi(z)\partial_z^\alpha w_\xi(z)R_{\sigma,1}(i,r,z^{-1}x,\xi) \mathop{dz}\\
     &\doteq I(i,r,x,\xi)+J(i,r,x,\xi),
 \end{align*}
 where $R_{\sigma,1}$ is the remainder in the Taylor expansion of order $1$ of $\sigma_A$, centered at $x$. Therefore, if we can prove that $I,J$ and $R_N$ all belong to the specified symbol class, we will have proved the lemma. Let is examine each one at a time.
 First, we see that the first term in the sums in 
 $I$ and $J$ (of order 0) vanishes since
 \begin{equation*}
     \int_G w_\xi(z)\partial_z^\alpha w_\xi(z)\mathop{dz}=0,
 \end{equation*}
 for $|\alpha|=1$ because the functions $w_\xi$ and $\partial^\alpha w_\xi$ are even and odd, respectively, by Lemmas \ref{evenodd}, \ref{int0} and \ref{lemmaw_xi}. In particular, $I$ is given by 
 \begin{align*}
     \sum_{2\leq|\alpha|\leq N}\Delta_\xi^\alpha\int_G w_\xi(z)\partial_z^\alpha w_\xi(z)&\mathop{dz}\sigma_A(i,r,x,\xi)
     \\
     &=\sum_{2\leq|\alpha|\leq N}\sum_{\kappa,\lambda,\mu}C_{\kappa,\lambda,
     \mu}\int_G (\Delta_\xi^\kappa w_\xi(z))(\Delta_\xi^\lambda\partial_z^\alpha w_\xi(z))\mathop{dz}\Delta_{\xi}^\mu\sigma_A(i,r,x,\xi),
 \end{align*}
 for some constants $C_{\kappa,\lambda,
     \mu}$, where the sum is taken over $|\kappa+\lambda+\mu|\geq|\alpha|$, where here we have used the ``Leibniz's" rule for difference operators, that is, Lemma \ref{leibniz}. Recalling that by Lemma \ref{lemmaw_xi} $(x,[\xi])\mapsto w_\xi(x)\text{Id}_{d_\xi}$ is in $\mathcal{S}_{\rho,(\rho+\delta)/2}^{d(\rho+\delta)/4}(G)$, and using that $\sigma_A\in\mathcal{S}_{\rho,\delta}^{m}(G\otimes\text{End}(E_0))$, we get that
 \begin{equation}
\int_G|(\Delta_\xi^\kappa w_\xi(z))(\Delta_\xi^\lambda\partial_z^\alpha w_\xi(z))|\mathop{dz}\lesssim \pseudoxi^{d\frac{(\rho+\delta)}{2}-\rho(|\kappa|+|\lambda|)+|\alpha|\frac{(\rho+\delta)}{2}}\pseudoxi^{-d\frac{(\rho+\delta)}{2}}=\pseudoxi^{-\rho(|\kappa|+|\lambda|)+|\alpha|\frac{(\rho+\delta)}{2}},
 \end{equation}
 where we have taken into account that the support of $z\mapsto w_\xi(z)$ is contained in a set of measure $\sim \pseudoxi^{-d\frac{(\rho+\delta)}{2}}$, by Lemma \ref{lemmaw_xi}, and that taking differences in $\xi$ does not increase the support in $z$. Thus we get
 \begin{align*}
       \| I(i,r,x,\xi)\|_{op}&\lesssim\sum_{2\leq|\alpha|\leq N}\sum_{\kappa,\lambda,\mu}C_{\kappa,\lambda,
     \mu}\pseudoxi^{-\rho(|\kappa|+|\lambda|)+|\alpha|\frac{(\rho+\delta)}{2}}\pseudoxi^{m-\rho|\mu|}\\
     &\lesssim\sum_{2\leq|\alpha|\leq N}\pseudoxi^{m-\rho|\alpha|+|\alpha|\frac{(\rho+\delta)}{2}}\lesssim\pseudoxi^{m-(\rho-\delta)},
 \end{align*}
 since
 \begin{align*}
     -\rho|\alpha|+|\alpha|\frac{(\rho+\delta)}{2} = -|\alpha|\frac{(\rho-\delta)}{2}\leq -(\rho-\delta).
 \end{align*}
  Applying a similar argument to $\Delta_\xi^{\gamma}\partial_x^{\beta}I(i,r,x,\xi)$ we obtain the respective decay estimates, which allow us to conclude that $I\in \mathcal{S}_{\rho,\frac{(\rho+\delta)}{2}}^{m-(\rho-\delta)}(G\otimes\text{End}(E_0))$, as desired. We now consider the term $J$. In this case, 
  \begin{equation*}
      \|J(i,r,x,\xi)\|_{op}\lesssim \sum_{1\leq|\alpha|\leq N}\sum_{\kappa,\lambda,\mu}\int_G\|(\Delta_\xi^\kappa w_\xi(z))(\Delta_\xi^\lambda\partial_z^\alpha w_\xi(z))(\Delta_\xi^{\mu}R_{\sigma,1}(i,r,z^{-1}x,\xi)\|_{op}\mathop{dz}
  \end{equation*}
  where again we have used the Leibniz's rule for difference operators, and the middle sum is over a finite set satisfying $|\kappa+\lambda+\mu|\geq |\alpha|$. Using Lemma \ref{lemmaw_xi} and the estimates for the remainder in the Taylor expansion, we get that
  \begin{align*}
      \|J(i,r,x,\xi)\|_{op}&\lesssim \sum_{1\leq|\alpha|\leq N}\sum_{\kappa,\lambda,\mu}\int_{\text{supp}(w_\xi)}\pseudoxi^{d\frac{(\rho+\delta)}{2}-\rho(|\kappa|+|\lambda|)+|\alpha|\frac{(\rho+\delta)}{2})}|z|^{1}\mathop{dz}\times\\
      &\times \max_{|\gamma|\leq 1}\|\|\Delta_\xi^{\mu}\partial_y^{\gamma}\sigma_A(i,r,y,\xi)\|_{L^\infty(G)_y}\|_{op}\\
      &\lesssim\sum_{1\leq|\alpha|\leq N}\pseudoxi^{d\frac{(\rho+\delta)}{2}-\rho(|\alpha|)+|\alpha|\frac{(\rho+\delta)}{2}-\frac{(\rho+\delta)}{2}-d\frac{(\rho+\delta)}{2}+m+\delta}\\
      &\lesssim\pseudoxi^{m-\frac{(\rho-\delta)}{2}-\frac{(\rho-\delta)}{2}}=\pseudoxi^{m-(\rho-\delta)}.
  \end{align*}
  Again, applying a similar argument to $\Delta_\xi^{\gamma}\partial_x^{\beta}J(i,r,x,\xi)$ we obtain the respective decay estimates, which allow us to conclude that $J\in \mathcal{S}_{\rho,\frac{(\rho+\delta)}{2}}^{m-(\rho-\delta)}(G\otimes\text{End}(E_0))$, as desired. It only remains necessary to study the remainder $R_N$. In this case, by Proposition \ref{propasymptotic}, we have that $R_N\in\mathcal{S}_{\rho,\frac{\rho+\delta}{2}}^{m-\frac{(\rho+\delta)}{2}(N+1)}(G\otimes\text{End}(E_0))$. By taking $N$ sufficiently large, we obtain that $R_N$ belongs to the desired symbol class, which concludes the proof.\end{proof}

\begin{proof}[Proof of Theorem \ref{maintheo}]
    Let $Q=A-P$, with the operator $P$ as in Lemma \ref{lemmaP}. Let $u\in C^\infty(G,E_0)$. Then $A=P+Q$ and the positivity of $P$ implies
    \begin{align*}
        \Real(Au,u)_{L^2(G,E_0)} &= \Real(Pu,u)_{L^2(G,E_0)}+\Real(Qu,u)_{L^2(G,E_0)}\\
        &\geq \Real(Qu,u)_{L^2(G,E_0)}.
    \end{align*}
    Let now $P_0=\text{Op}(\{p(i,r,x,x,\xi)_{i,r=1}^n\})$. Writing $Q = (A-P_0)+(P_0-P)$, we have
    \begin{equation*}
        \sigma_{A-P_0}(i,r,x,\xi)=\sigma_A(i,r,x,\xi)-p(i,r,x,x,\xi)
    \end{equation*}
    and
    \begin{equation*}
        \sigma_{P_0-P}(i,r,x,\xi)=p(i,r,x,x,\xi)-\sigma_P(i,r,x,\xi).
    \end{equation*}
    Consequently, both $A-P_0$ and $P_0-P$ are bounded from $H^{\frac{m-(\rho-\delta)}{2}}(G,E_0)$ to $H^{-\frac{m-(\rho-\delta)}{2}}(G,E_0)$ by Lemmas \ref{lemma1} and \ref{lemma2}, respectively. It follows that $Q$ is also bounded between these spaces so that
    \begin{align*}
        |\Real(Qu,u)_{L^2(G,E_0)}|&\leq\|Qu\|_{H^{-\frac{m-(\rho-\delta)}{2}}(G,E_0)}\|u\|_{H^{\frac{m-(\rho-\delta)}{2}}(G,E_0)}\\
        &\lesssim \|u\|_{H^{\frac{m-(\rho-\delta)}{2}}(G,E_0)}^2,
    \end{align*}
    completing the proof of Theorem \ref{maintheo}.
\end{proof}
This result yields the sharp G\aa rding inequality for  homogeneous vector bundles over compact manifolds as a corollary. Before stating it precisely, we will recall the quantisation on homogeneous vector bundles as introduced in the setting of compact homogeneous manifolds, \cite{DuvanHomo}. But first, we recall the setting of compact homogeneous manifolds, also  following \cite{Wallach}.
\subsection{Sharp G\aa rding inequality on compact homogeneous vector bundles}
Let $p:E\to X$ be a vector bundle. A continuous map $s:X\to E$ is called a section of $E$ if for all $x\in X$, $p(s(x)) = p(x)$. We denote by $\Gamma (E)$ the set of all sections of $E$. If $X,E$ are smooth manifolds, we also define $\Gamma^\infty(E)$ the set of all {\it{smooth}} sections of $E$. If $X$ is orientable, the space $L^q(E)$, $1\leq q<\infty$, is then defined as the completion of the set of all smooth sections $s\in \Gamma^\infty(E)$ such that
\begin{equation}
\|s\|_{L^q(E)}\doteq\left(\smallint_X\|s(x)\|_{E_x}^q dx\right)^{\frac{1}{q}}<\infty.
\end{equation}
Now, consider $G$ be a compact Lie group and $K$ a closed subgroup of $G$. Let $M=G/K$ be equipped with its natural compact manifold topology. There exists a natural left action of $G$ on $M$ given by $g\cdot hK=ghK$, for every $g,h\in G$. We say that a vector bundle $p:E\to M$ is a homogeneous vector bundle over $M$ if $G$ acts on $E$ on the left and this action satisfies:
\begin{enumerate}
    \item $g\cdot E_x = E_{gx}$, for all $x\in M$, $g\in G$.
    \item The previously induced mappings from $E_x$ to $E_{gx}$ are linear.
\end{enumerate}
There is natural left action of $G$ on $\Gamma(E)$, $G\times \Gamma(E)\to \Gamma(E)$ given by
\begin{equation*}
    (g\cdot s)(x) = g\cdot s(g^{-1}x),
\end{equation*}
for all $x\in X$, $g\in G$.
Consider $p:E\to M$. Let $E_0=p^{-1}(K)$, be the fiber at the identity coset. As proved in \cite{Bott}, there exists $\tau\in\text{Hom}(K,\text{End}(E_0))$, and a natural right action of $K$ on $G\times E_0$ by $(g,v) = (gk,\tau^{-1}(k)v)$. We denote by $G\times_{\tau} E_0$ the quotient $(G\times E_0)/K$ under this action. In \cite{Bott}, Bott shows that $G\times_{\tau} E_0$ admits a natural homogeneous vector bundle structure, and that in fact there exists $\tau\in\text{Hom}(K,\text{End}(E_0))$ such that 
\begin{equation*}
     E\cong G\times_{\tau} E_0
\end{equation*}
as vector bundles. Now consider the vector subspace $C^{\infty}(G,E_0)^{\tau}\subset C^{\infty}(G,E_0)$ given by
\begin{equation*}
    C^{\infty}(G,E_0)^{\tau}=\{f\in C^\infty(G,E_0)|\forall g\in G,\forall k\in K, f(gk)=\tau(k)^{-1}f(g)\},
\end{equation*}
and likewise $L^2(G,E_0)^{\tau}\subset L^2(G,E_0)$ by
\begin{equation*}
    L^2(G,E_0)^{\tau}=\{f\in L^2(G,E_0)|  f(gk)=\tau(k)^{-1}f(g) \text{ for a. e. }g,k\in G\}.
\end{equation*}
It can be shown that the bijection $\chi_\tau:\Gamma^\infty(E)\to C^\infty(G,E_0)^\tau$, given by 
\begin{equation*}
    \chi_\tau(s)(g)\doteq g^{-1}\cdot s(gK)
\end{equation*}
for every $g\in G$ extends to a surjective isometry from $L^2(E)$ into $L^2(G,E_0)^{\tau}$, so that we may identify $\Gamma^\infty(E)\cong C^{\infty}(G,E_0)^{\tau}$, $L^2(E)\cong L^2(G,E_0)^{\tau}$. The Sobolev space $H^s(E)$ for $s\in \R$ is then defined as the completion of the set of smooth sections under the norm
\begin{equation*}
    \|u\|_{H^s(E)}\doteq\|\chi_\tau u\|_{H^s(G,E_0)}.
\end{equation*}
Now let $\Tilde{A}:\Gamma^\infty(E)\to\Gamma^\infty(E)$ be a continuous linear operator. Then it induces a continuous linear map $A:C^\infty(G,E_0)^\tau\to C^\infty(G,E_0)^\tau$ by
\begin{equation*}
    A=\chi_\tau\circ \Tilde{A}\circ\chi_\tau^{-1}.
\end{equation*}
If $A\in\Psi_{\rho,\delta}^m(G\otimes \text{End}(E_0))$, we say that $\Tilde{A}\in \Psi^m_{\rho,\delta}(E)$, and define its symbol by $\sigma_{\Tilde{A}}\doteq \sigma_A$, where $\sigma_A$ is the matrix-valued symbol defined in \eqref{symbol}. The quantisation formula \eqref{quantisation1} then implies 
\begin{align*}
    \Tilde{A}s(gK)=\chi_{\tau}^{-1}\left(\sum_{i,r=1}^{d_\tau}\sum_{[\xi]\in\widehat{G}}d_\xi\Tr\left(\xi(x)\sigma_{\tilde{A}}(i,r,g,\xi)\widehat{\chi_\tau s}(i,\xi)\right)e_{i}\right).
\end{align*}

\begin{corollary}\label{coro1}
    Let $p:E\to M=G/K$ be a homogeneous vector bundle over a compact homogeneous manifold $M$, where $K<G$ are compact Lie groups,  $E\cong G\times_{\tau} E_0$. Let $\Tilde{A}\in \Psi^m_{\rho,\delta}(E)$, $0\leq\delta<\rho\leq 1$, be such that its matrix-valued symbol $\sigma_A(x,\xi)$ is positive semi-definite for every $(x,[\xi])\in G\times \widehat{G}$, in the sense of Theorem \ref{maintheo}. Then there exists $C>0$ such that
    \begin{equation*}
        \Real(\Tilde{A}s,s)_{L^2(E)}\geq -C\|s\|^2_{H^{\frac{m-(\rho-\delta)}{2}}(E)},
    \end{equation*}
     for every $s\in\Gamma^\infty(E)$.
\end{corollary}
\begin{proof}
    Let $A=\text{Op}(\sigma_A)$ be the vector-valued operator associated with $\Tilde{A}$. For $s\in\Gamma^\infty(E)$, set $u=\chi_{\tau}s$. Theorem \ref{maintheo} implies that there exists $C>0$ such that
    \begin{align*}
        \Real(Au,u)_{L^2(G,E_0)}&\geq -C\|u\|^2_{H^{\frac{m-(\rho-\delta)}{2}}(G,E_0)}\\
        &=-C\|s\|^2_{H^{\frac{m-(\rho-\delta)}{2}}(E)}.
    \end{align*}
    The result then follows from the fact that
    \begin{equation*}
\Real(\Tilde{A}s,s)_{L^2(E)}=\Real(\chi_\tau^{-1}A\chi_{\tau}\chi_\tau^{-1}u,\chi_{\tau}^{-1}u)_{L^2(E)}=\Real(A\chi_{\tau}\chi_\tau^{-1}u,u)_{L^2(G,E_0)}=\Real(Au,u)_{L^2(G,E_0)},
    \end{equation*}
where we have used that $\chi_{\tau}^{-1}:L^2(G,E_0)^{\tau}\to L^2(E)$ is an isometry.
\end{proof}
Finally, note that if $E=G\times_{\widehat{1}}\mathbb{C}\cong M$ is the trivial bundle, then we can identify $\Gamma^\infty(E)\cong C^\infty(M)$, and $\chi_{\widehat{1}}:C^\infty(M)\to C^\infty(G)^K$ is just the projective lifting
\begin{equation*}
    \chi_{\widehat{1}}f(g)\equiv\dot{f}(g)\doteq f(gK),\,\forall g\in G.
\end{equation*}
The pseudo-differential operator classes and Sobolev spaces of $M$ are then defined likewise.
\begin{corollary}\label{coro2}
    Let $M=G/K$ be a compact homogeneous manifold, where $K<G$ are compact Lie groups. Let $\Tilde{A}\in\Psi^{m}_{\rho,\delta}(M),\,0\leq\delta<\rho\leq 1$, be such that its matrix valued symbol $\sigma_A=\sigma_A(x,\xi)\in \mathcal{S}^{m}_{\rho,\delta}(G)$ is positive semi-definite for every $(x,[\xi])\in G\times \widehat{G}$. Then there exists $C>0$ such that
    \begin{align*}
        \Real(\Tilde{A}u,u)_{L^2(M)}&\geq -C\|u\|^2_{H^{\frac{m-(\rho-\delta)}{2}}(M)},
    \end{align*}
    for every $u\in C^\infty(M).$
\end{corollary}

\section{Application for vector-valued diffusion problems}\label{PDE}

In this section we will apply our results to study well-posedness of pseudo-differential evolution Cauchy problems. Let $T>0$, we are interested in proving existence and uniqueness for the following Cauchy problem
\begin{align}\label{Cauchyproblem}
    \begin{cases}
        \frac{\partial v}{\partial t}-K(t,x)v=f,\quad v\in\mathcal{D}'((0,T)\times G,E_0),\\
        v(0)=u_0,
    \end{cases}
\end{align}
where $u_0\in L^2(G)$, $f\in L^2([0,T],L^2(G,E_0))$, $K(t)=K(t,x)\in\Psi^{m}_{\rho,\delta}(G\otimes \text{End}(E_0))$, $0\leq \delta<\rho\leq1$, $m\leq \rho-\delta$. For that, we will assume appropriate conditions on $K(t)$ in order to apply our previous results. We begin with the following proposition:
\begin{prop}\label{propestimate}
    Let $K(t)=K(t,x)\in\Psi^{m}_{\rho,\delta}(G\otimes \text{End}(E_0))$, for every $0\leq t\leq T$, and some $0\leq \delta<\rho\leq1$, be a pseudo-differential operator of order $m\leq \rho-\delta$, such that the matrix valued symbol $\sigma_{-K}(t,x,\xi)$ is positive semi-definite in the sense of Theorem \ref{maintheo}, and depends continuously on $t$. Then if $v\in C^{1}([0,T],L^2(G,E_0))\cap C([0,T],H^{\frac{m-(\rho-\delta)}{2}})$, then there exist $C',C''>0$ such that
\begin{equation}\label{energy1} 
\|v(t)\|_{L^2(G,E_0))}^2\leq\left(C'\|v(0)\|_{L^2(G,E_0))}^2+C''\int_0^T\|Q(t)v(t)\|_{L^2(G,E_0)}^2dt\right),
\end{equation}
and
\begin{equation}\label{energy2} 
\|v(t)\|_{L^2(G,E_0))}^2\leq\left(C'\|v(T)\|_{L^2(G,E_0))}^2+C''\int_0^T\|Q(t)^*v(t)\|_{L^2(G,E_0)}^2dt\right),
\end{equation}
for every $0\leq t\leq T$, where $Q(t)=\partial_t-K(t)$, and $Q(t)^*$ denotes the adjoint of $Q(t)$.
\end{prop}
\begin{proof}
    First, define
    \begin{equation*}
        f(t)\doteq Q(t)v(t),
    \end{equation*}
    where $Q(t)=\partial_t-K(t),\,0\leq t\leq T$. Then 
    \begin{align*}
        \frac{d}{dt}\|v(t)\|_{L^2(G,E_0)}^2&=\left(\frac{dv(t)}{dt},v\right)_{L^2(G,E_0)}+\left(v,\frac{dv(t)}{dt}\right)_{L^2(G,E_0)}\\
        &=\left(f(t)+K(t)v(t),v\right)_{L^2(G,E_0)}+\left(v,f(t)+K(t)v(t)\right)_{L^2(G,E_0)}\\
        &=2\Real(K(t)v(t),v(t))+2\Real(f(t),v(t)).
    \end{align*}
    From the parallelogram law 
    \begin{equation*}
        2\Real(f(t),v(t))\leq 2\|f(t)\|_{L^2(G,E_0)}^2+2\|v(t)\|_{L^2(G,E_0)}^2.
    \end{equation*}
    From the sharp Garding inequality \eqref{ineqtheo}, and the continuity of the symbol on $t\in[0,T]$: 
    \begin{equation*}
        \Real(-Kv,v)_{L^2(G,E_0)}^2\geq -C\|v\|^2_{H^{\frac{m-(\rho-\delta)}{2}}(G,E_0)}\geq -C\|v\|^2_{L^2(G,E_0)},
    \end{equation*}
    where the last inequality comes from our restriction on $m$ and the continuous inclusion $ L^2(G,E_0)\hookrightarrow H^{-\epsilon}(G,E_0)$. Together, these inequalities then imply
    \begin{equation*}
        \frac{d}{dt}\|v(t)\|_{L^2(G,E_0)}^2\leq C'\|v(t)\|_{L^2(G,E_0)}^2+C''\|f(t)\|_{{L^2(G,E_0)}}^2,
    \end{equation*}
    for $C'=(C+2), C''=2>0$. Applying Gronwall's Lemma to this inequality we obtain the desired estimate. The other estimate follows analogously, applying the previous analysis to $v(T-\cdot)$ instead of $v(\cdot)$, $f(T-\cdot)$ instead of $f(\cdot)$, and $Q(t)^*=-\partial_t-K(t)^*$ instead of $Q(t)$ and using that $\Real(K(T-t)^*)=\Real(K(T-t))$.
\end{proof}
\begin{theorem}\label{theoexist}
    Let $u_0\in L^2(G,E_0)$, $f\in L^2([0,T],L^2(G,E_0))$. Let also $K(t)$ satisfy the conditions of Proposition \ref{propestimate}. Then there exists a unique $v\in C^1([0,T],L^2(G,E_0))\cap C([0,T],$\\$H^{\frac{m-(\rho-\delta)}{2}}(G,E_0))$ which solves the Cauchy problem \eqref{Cauchyproblem}. Moreover, $v$ satisfies the energy estimate
    \begin{equation}\label{energy} \|v(t)\|_{L^2(G,E_0))}^2\leq\left(C'\|v(0)\|_{L^2(G,E_0))}^2+C''\int_0^T\|f(t)\|_{L^2(G,E_0)}^2dt\right),
\end{equation}
for every $0\leq t\leq T$.
\end{theorem}
\begin{proof}
 We have already proved the estimate \eqref{energy} for $v\in X_m\doteq C^1([0,T],L^2(G))\cap$$ \\C([0,T],H^{\frac{m-(\rho-\delta)}{2}}(G,E_0))$. We will now prove that the solution exists and belongs to this space. Let us again denote $Q(t)\doteq \partial_t-K(t)$. Consider the space
 \begin{equation*}
     E\doteq\{\phi\in C^1([0,T],L^2(G))\cap C([0,T],H^{\frac{m-(\rho-\delta)}{2}}(G,E_0))|\phi(T)=0\},
 \end{equation*}
 and also $Q^{*}E\doteq \{Q^{*}\phi|\phi\in E\}$. Define the linear functional $\beta\in (Q^{*}E)' $ by 
 \begin{equation*}
\beta(Q^{*}\phi)\doteq\int_{0}^{T}(\phi(t),f(t))_{L^2(G,E_0)}dt+(\phi(0),u_0)_{L^2(G,E_0)}.
 \end{equation*}
It follows from \eqref{energy2} that for every $\phi\in E$:
\begin{equation}\label{ineqphi}
    \|\phi(t)\|_{L^2(G,E_0)}^2\leq C''\int_0^T\|Q(t)^*\phi(t)\|_{L^2(G,E_0)}^2dt.
\end{equation}
Hence
\begin{align*}
    |\beta(Q^*\phi)|&\leq \int_0^T\|f(t)\|_{L^2(G,E_0)}\|\phi(t)\|_{L^2(G,E_0)}dt+\|u_0\|_{L^2(G,E_0)}\|\phi(0)\|_{L^2(G,E_0)}\\
    &\leq \|f\|_{L^2([0,T],L^2(G,E_0))}\|\phi\|_{L^2([0,T],L^2(G,E_0))}+\|u_0\|_{L^2(G,E_0)}\|\phi(0)\|_{L^2(G,E_0)}\\
    &\lesssim_T \left(\|f\|_{L^2([0,T],L^2(G,E_0))}+\|u_0\|_{L^2(G,E_0)}\right)\|Q^*\phi\|_{L^2([0,T],L^2(G,E_0))},
\end{align*}
where the last inequality is partly due to the estimate \eqref{ineqphi}. This shows that $\beta$ is a bounded functional on $\mathcal{I}\doteq Q^*E\cap L^2([0,T],L^2(G,E_0))$, with the topology induced by the usual norm in $L^2([0,T],L^2(G,E_0))$. The Hahn-Banach extension theorem then implies that we may extend $\beta$ to a bounded functional $\Tilde{\beta}$ defined on all $L^2([0,T],L^2(G,E_0))$, and by the Riesz representation theorem for Hilbert spaces, there exists $v\in L^2([0,T],L^2(G))$ such that
\begin{equation*}
    \Tilde{\beta}(\psi)=(\psi,v)_{L^2([0,T],L^2(G,E_0))},
\end{equation*}
for every $\psi\in L^2([0,T],L^2(G,E_0))$. In particular, for $\psi=Q^*\phi\in\mathcal{I}$, we have
\begin{equation*}
    \beta(Q^*\phi)=\Tilde{\beta}(Q^*\phi)=(Q^*\phi,v)_{L^2([0,T],L^2(G,E_0))}.
\end{equation*}
Notice that we can view $C^\infty_0((0,T),C^\infty(G))$ as a subspace of $E$, that is
\begin{equation*}
    C^\infty_0((0,T),C^\infty(G))\subset \{\phi\in C^1([0,T],L^2(G))\cap C([0,T],H^{\frac{m-(\rho-\delta)}{2}}(G,E_0))|\phi(T)=0\},
\end{equation*}
and so we have the identity
\begin{align*}
    (\phi,f)_{L^2([0,T],L^2(G,E_0))}&=\int_0^T(\phi(t),f(t))_{L^2(G,E_0)}dt\\
    &=\int_0^T(\phi(t),f(t))_{L^2(G,E_0)}dt+(\phi(0),u_0)_{L^2(G,E_0)}\\
    &=\beta(Q^*\phi)\\
    &=(Q^*\phi,v)_{L^2([0,T],L^2(G,E_0))},
\end{align*}
for every $\phi\in C^\infty_0((0,T),C^\infty(G))$, as $\phi(0)\equiv 0$. This implies $v\in\text{Dom}(Q^{**})$, and from $Q^{**}=Q$ we obtain
\begin{equation*}
    (\phi,Qv)=(Q^*\phi,v)=(\phi,f),\,\forall \phi\in C^\infty_0((0,T),C^\infty(G)).
\end{equation*}
This implies $Qv=f$. Moreover, we claim that $v(0)=u_0$.
Indeed, note that for $\phi\in E$, integration by parts yields
\begin{align*}
    \int_0^T\left({\phi(t)},\frac{dv(t)}{dt}\right)dt= -(\phi(0),v(0))-\int_0^T\left(\frac{d\phi(t)}{dt},v(t)\right)dt.
\end{align*}
On the other hand, it is also true that
\begin{align*}
    \int_0^T\left(\phi(t),\frac{dv(t)}{dt}\right)dt&=\int_0^T(\phi(t),f(t)+K(t)v(t))dt\\
    &=\int_0^T(\phi(t),f(t))dt+\int_0^T(K(t)^*\phi(t),v(t))dt.
\end{align*}
Combining these two, we get that
\begin{equation*}
   \int_0^T(\phi(t),f(t))dt+\int_0^T(K(t)^*\phi(t),v(t))dt= -(\phi(0),v(0))-\int_0^T\left(\frac{d\phi(t)}{dt},v(t)\right)dt,
\end{equation*}
so that
\begin{align*}
    \int_0^T(\phi(t),f(t))dt+(\phi(0),v(0))&=\int_0^T\left(-\frac{d\phi(t)}{dt}-K(t)^*\phi(t),v(t)\right)dt\\
    &=\int_0^T(Q(t)^*\phi(t),v(t))dt\\
    &=(Q^*\phi,v)=\beta(Q^*\phi)\\
    &=\int_0^T(\phi(t),f(t))dt+(\phi(0),u_0),
\end{align*}
and so $v(0)=u_0$. Finally, to prove the uniqueness, assume that $u\in C^1([0,T],L^2(G,E_0))\cap C([0,T],H^{\frac{m-(\rho-\delta)}{2}(G,E_0)})$ is also a solution to \eqref{Cauchyproblem}. Then the function $w\doteq v-u$ solves the Cauchy problem
\begin{equation*}
    \begin{cases}
        \frac{dw(t)}{dt}-K(t,x)w=0,\\
        w(0)=0,
    \end{cases}
\end{equation*}
and from Proposition \ref{propestimate}, $\|w(t)\|_{L^2(G,E_0)}=0$ for all $0\leq t\leq T$. Hence we have that $v(t,x)=u(t,x)$ for all $t\in[0,T]$ and a.e. $x\in G$. 
\end{proof}
We may apply the ideas of Corollaries \ref{coro1} and \ref{coro2} to obtain the following corollary to Theorem \ref{theoexist}.
\begin{corollary}
Let $M=G/K$ be a compact homogeneous manifold. Let $u_0\in L^2(M)$, $f\in L^2([0,T],L^2(M))$. Let $\tilde{K}(t)=\tilde{K}(t,x)\in\Psi^{m}_{\rho,\delta}(M)$, for every $0\leq t\leq T$, and some $0\leq \delta<\rho\leq1$, be a pseudo-differential operator of order $m\leq \rho-\delta$, such that the matrix valued symbol $\sigma_{-K}(t,x,\xi)$ is positive semi-definite for all $t,x,\xi$, and depends continuously on $t$. Then there exists a unique $v\in C^1([0,T],L^2(M))\cap C([0,T],H^{\frac{m-(\rho-\delta)}{2}}(M))$ which solves the Cauchy problem:
\begin{align}
    \begin{cases}
        \frac{\partial v}{\partial t}-K(t,x)v=f,\\
        v(0,x)=u_0(x).
    \end{cases}
\end{align}
Moreover, $v$ satisfies the energy estimate
    \begin{equation} \|v(t)\|_{L^2(M)}^2\leq\left(C'\|v(0)\|_{L^2(M)}^2+C''\int_0^T\|f(t)\|_{L^2(M)}^2dt\right),
\end{equation}
for every $0\leq t\leq T$.
\end{corollary}
We exemplify an application of Theorem \ref{theoexist} to prove existence and uniqueness to the following system of differential equations.\\

        Denote by $\Xi=(I+\mathcal{L}_G)^{1/2}\in\Psi^1_{1,0}(G)$ the Bessel potential of order $1$.
        Consider the vector-valued Cauchy problem given by
\begin{equation}\label{CauchyExample}
    \begin{cases}
        \mathlarger{\frac{d}{dt}} \begin{pmatrix}
            u_1(t) \\ u_2(t)
        \end{pmatrix}
        -
        \begin{pmatrix}
         -\Xi&-\Xi\\
         -\Xi&-\Xi
        \end{pmatrix}
        \begin{pmatrix}
            u_1(t) \\ u_2(t)
        \end{pmatrix}
        =
        \begin{pmatrix}
            0 \\ 0
        \end{pmatrix},
        \\
        \begin{pmatrix}
            u_1(0) \\ u_2(0)
        \end{pmatrix}
        =\begin{pmatrix}
            u_{10} \\ u_{20}
        \end{pmatrix}.
    \end{cases}
\end{equation}
Viewing the matrix $K(t)\equiv K=\begin{pmatrix}
        -\Xi&-\Xi\\
         -\Xi&-\Xi
        \end{pmatrix}$ as pseudo-differential operator in the class $\Psi^1_{1,0}(G\otimes \text{End}(\mathbb{C}^2))$, its symbol is given by
        \begin{equation*}
           \sigma_{-K}(x,\xi)= \begin{pmatrix}
         \langle\xi\rangle I_{d_\xi}&\langle\xi\rangle I_{d_\xi}\\
         \langle\xi\rangle I_{d_\xi}&\langle\xi\rangle I_{d_\xi}
        \end{pmatrix},
        \end{equation*}
        for every $(x,[\xi])\in G\times \widehat{G}$. It is clear that in this case $\sigma_{-K}$ satisfies the previous hypothesis and so the last theorem holds true for the Cauchy problem \eqref{CauchyExample}. With a simple computation one may find a solution given by
        \begin{equation}
    \begin{cases}\label{solution}
            \widehat{u_1}(t,\xi)=\frac{1}{2}e^{-2\langle\xi\rangle t}({\widehat{u_{10}}(\xi)+\widehat{u_{20}}(\xi)})-\frac{1}{2}({\widehat{u_{10}}(\xi)-\widehat{u_{20}}(\xi)}) \\ \widehat{u_2}(t,\xi)=\frac{1}{2}e^{-2\langle\xi\rangle t}({\widehat{u_{10}}(\xi)+\widehat{u_{20}}(\xi)})+\frac{1}{2}({\widehat{u_{10}}(\xi)-\widehat{u_{20}}(\xi)}).   
        \end{cases}
        \end{equation}
        Theorem \ref{theoexist} then implies that this is the unique solution to \eqref{CauchyExample}, and that it satisfies the estimates
        \begin{equation*}
            \|u_1(t)\|_{L^2(G)}^2+\|u_2(t)\|_{L^2(G)}^2\leq C'(\|u_{10}\|_{L^2(G)}^2+\|u_{20}\|_{L^2(G)}^2),
        \end{equation*}
        for some $C'>0$. Note that indeed, by squaring \eqref{solution}, adding both equations and applying Plancherel's Identity, one can easily obtain that the above inequality holds for $C'=\frac{1}{2}.$

\end{document}